\documentclass{article}

\usepackage{mathrsfs}	
\usepackage[all]{xy}
\usepackage{amsmath}
\usepackage{amsfonts}
\newtheorem{remark}{Remark}
\newtheorem{theorem}{Theorem}
\newtheorem{proposition}{Proposition}
\newtheorem{lemma}{Lemma}

\def\CQFD{$\Box$}
\newenvironment{proof}{{\it Proof}.}{\hfill \CQFD \\}

\def\NN{\mathbb{N}}
\def\ZZ{\mathbb{Z}}
\def\kk{\mathbb{K}}
\def\PP{\mathbb{P}}
\def\Sc{\mathscr{S}}
\def\Hc{\mathscr{H}}
\def\Zc{\mathcal{Z}}

\def\Pc{\mathscr{P}}

\def\coker{\mathrm{coker}}
\def\Proj{\mathrm{Proj}}
\def\deg{\mathrm{deg}}
\def\lcm{\mathrm{lcm}}
\def\Sym{\mathrm{Sym}}

\def\ann{\mathrm{ann}}
\def\sat{\mathrm{sat}}
\def\indeg{\mathrm{indeg}}
\def\Hom{\mathrm{Hom}}
\def\Homgr{\mathrm{Homgr}}
\def\im{\mathrm{im}}
\def\length{\mathrm{length}}

\def\qc{\mathfrak{q}}
\def\pp{\mathfrak{p}}
\def\mm{\mathfrak{m}}

\long\def\symbolfootnote[#1]#2{\begingroup%
\def\thefootnote{\fnsymbol{footnote}}\footnote[#1]{#2}\endgroup}
\title{Implicitization of Bihomogeneous Parametrizations of Algebraic Surfaces via Linear Syzygies}

\author{
Laurent Bus\'e, Marc Dohm }

\begin{document}

\maketitle

\begin{abstract}

We show that the implicit equation of a surface in 3-di\-men\-sio\-nal 
projective space parametrized by bi-homogeneous polynomials of bi-degree $(d,d)$, for a given
integer $d\geq 1$, can be represented and computed from the linear
syzygies of its parametrization if the base points are isolated
and form locally a complete intersection.

\end{abstract}

\section{Introduction}

\symbolfootnote[0]{Copyright ACM, 2007. This is the author's version of the work. It is posted here by permission of ACM for your personal use. Not for redistribution. The definitive version was published in 
Proceedings of the 2007 international Symposium on Symbolic and Algebraic Computation (Waterloo, Ontario, Canada, July 29 - August 01, 2007), ACM Press, New York, NY, pages 69-76, http://doi.acm.org/10.1145/1277548.1277559.}

Implicitization, i.e. finding the implicit equation of an
algebraic curve or surface defined parametrically, is a classical
problem and there are numerous approaches to its solution, most of
them based either on resultants, Gr\"{o}bner bases, or syzygies. A
good historical overview of methods based on resultants or
Gr\"{o}bner bases can be found in \cite{SC95} and \cite{Co01}.

\medskip
Syzygy-based methods - also known as ``moving surfaces'' methods -
were introduced in \cite{SC95} and have been further developed in
a number of publications. In the case of curves, these methods
show that the linear syzygies of a given parametrization can be
put together in a \emph{square} matrix whose determinant is an implicit
equation. Several attempts to generalize these results to the case
of surfaces have been made. The construction of a square
matrix whose determinant is the implicit equation requires the use
of quadratic syzygies in addition to the linear syzygies.
Moreover, these methods are only valid for birational
parametrizations and need additional assumptions in the
presence of base points.

Just to name some recent publications on sy\-zy\-gy-ba\-sed methods, we
cite \cite{Co03}, which treats the implicitization of
base-point-free homogeneous parametrizations and \cite{BCD03},
which does the same for parametrizations with base points. In
\cite{AHW05} a determinantal representation of the implicit
equation of a bi-homogeneous parametrization is constructed with
linear and quadratic relations, whereas \cite{KD06} gives such a
construction in the toric case.

\medskip

Recently it has been proved in  \cite{BuJo03} and \cite{BuCh05}
that surfaces parametrized by the projective plane can be
represented and computed only using the linear syzygies of the
parametrization, in the case where the base points are isolated
and locally complete intersections. In some sense, this result is
a natural generalization of the method of ``moving curves''
developed for planar curves in \cite{SC95}; the only difference is
that the matrix obtained in the case of surfaces is not square,
but still represents the surface (see the end of Section \ref{approx} for a
detailed explanation of this term).

\medskip

In this paper our main objective is to develop a similar
implicitization technique for surfaces given by bi-ho\-mo\-ge\-ne\-ous
parametrizations, which are of interest for a number of
applications in geometric modelling and computer-aided design. We
will show that also in this case the surface can be represented by
a non-square matrix constructed by only using linear syzygies and
we will explain how to efficiently compute this matrix with
standard computer algebra systems.

\medskip

More precisely, we focus on the following problem. Let $\kk$ be any
field (all the varieties we will consider hereafter are understood
to be taken over $\kk$). We suppose given a rational map
\begin{eqnarray*}
 \PP^1\times \PP^1 & \xrightarrow{\phi} & \PP^3 \\
(s:u)\times (t:v) & \mapsto &
(f_1:f_2:f_3:f_4)(s,u,t,v)
\end{eqnarray*}
where each polynomial $f_1,f_2,f_3,f_4$ is bi-homogeneous of
bi-degree $(d,d)$, $d$ being a given positive integer, with
respect to the homogeneous variables $(s:u)$ and
$(t:v)$. We assume that
\begin{itemize}
\item $\phi$ parametrizes a surface $\Hc$ (which is equivalent to
require that $\phi$ is a generically finite map onto its image)
which is hence irreducible
\item the greatest common divisor of
$f_1,f_2,f_3,f_4$ is a non-zero constant which essentially
requires the number of base points of $\phi$ to be finite
(possibly zero).
\end{itemize}
We aim to find a representation of $\Hc$ in terms of linear
syzygies of $f_1,f_2,f_3$ and $f_4$ similar to the known ones for
plane curves and for space surfaces parametrized by the projective
plane.

\medskip

The paper is organized as follows. In Section \ref{segre} we give
an equivalent formulation of our problem which replaces the given
$\NN\times\NN$-graduation by a single $\NN$-graduation. In Section
\ref{approx} we will introduce an associated approximation complex
that will be used in Section \ref{mainresult} to prove our main
result. Then an algorithmic version is detailed in Section \ref{algo}, as well as an illustrative example.

\section{The Segre embedding}\label{segre}

It is well-known that $\PP^1\times \PP^1$ can be embedded in $\PP^3$ through the so-called \emph{Segre embedding}
\begin{eqnarray*}
 \PP^1\times \PP^1 & \xrightarrow{\rho} & \PP^3 \\
(s:u)\times(t:v) & \mapsto & (st:sv:ut:uv).   
\end{eqnarray*}
We denote by $\Sc$ its image, which is an irreducible surface of degree 2 in $\PP^3$, whose equation
in the coordinates $X_1,X_2,X_3,X_4$ of $\PP^3$ is known to be $X_1X_4-X_2X_3$.
Our strategy to solve our implicitization problem is to re\-pa\-ra\-me\-tri\-ze the
surface $\Hc$ by $\Sc \subset \PP^3$, that is to say to consider $\Hc$ as the closed image of the map
$\psi$ from $\Sc$ to
$\PP^3$ fitting in the commutative diagram
\begin{equation}
\xymatrix{
\PP^1\times \PP^1 \ar[r]^-{\phi} \ar[d]^{\rho} & \PP^3 \\
\Sc \ar[ur]_\psi }       \label{diagram}
\end{equation}

In the rest of this paper we will use the map $\psi = \phi \circ \rho^{-1}$ to implicitize $\Hc$,
which has the advantage of replacing the $\NN\times\NN$-graduation of $\PP^1\times \PP^1$ by a single $\NN$-graduation.
In order to justify this approach we need to describe explicitly the algebraic counterparts of the
maps in the above diagram.

\medskip

We begin with the map $\phi$. The polynomial ring $\kk[s,u]$ is canonically
$\NN$-graded,
$$\kk[s,u]=\bigoplus_{n\in\NN} \kk[s,u]_n=\kk[s,u]_0\oplus \kk[s,u]_1 \oplus \kk[s,u]_2\oplus \ldots$$
where $\kk[s,u]_i$ denotes the degree $i$ homogeneous
component of $\kk[s,u]$, and its homogeneous spectrum is the
projective line, i.e.~$\Proj(\kk[s,u])=\PP^1_\kk$. Of
course, the same is true for the polynomial ring $\kk[t,v]$. Now,
consider the $\NN$-graded $\kk$-algebra
$$S:=\bigoplus_{n\in \NN} \left(\kk[s,u]_n\otimes_\kk \kk[t,v]_n\right)  \subset \kk[s,u]\otimes_\kk \kk[t,v]$$
which is finitely generated by $S_1$ as an $S_0$-algebra.  Then $\PP^1\times \PP^1$ is the homogeneous spectrum $\Proj(S)$
of $S$. Introducing new indeterminates $T_1,T_2,T_3,T_4$, the map $\phi$ is hence induced by the graded $k$-algebra morphism
\begin{eqnarray*}
 \kk[T_1,T_2,T_3,T_4] & \xrightarrow{p} & S \\
T_i & \mapsto & f_i(s,u,t,v) \ \ i=1,\ldots,4.
\end{eqnarray*}
By \cite[Theorem 2.1]{BuJo03}, $\ker(p) \subset \kk[T_1,T_2,T_3,T_4]$ is the defining ideal of
the closed image of $\phi$ in $\PP^3=\Proj(\kk[T_1,\ldots,T_4])$;
it is prime (since $S$ is a domain) and principal (since it is of codimension one by
hypothesis and $\kk[T_1,T_2,T_3,T_4]$ is factorial), i.e. any generator of $\ker(p)$ gives an equation
of $\Hc$.

We now turn to the Segre embedding $\rho$. As we did for the map
$\phi$ (note that the Segre embedding is itself a parametrization
of a surface in projective space) the map $\rho$ is induced by the graded
$k$-algebra morphism
\begin{eqnarray*}
\kk[X_1,X_2,X_3,X_4] & \xrightarrow{\theta} & S \\
X_1 & \mapsto & st \\
X_2 & \mapsto & sv \\
X_3 & \mapsto & ut \\
X_4 & \mapsto & uv.
\end{eqnarray*}
However, in this case $\theta$ is surjective and graded (it preserves the degree). Moreover, it is easy to see\footnote{
We clearly have $(X_1X_4-X_2X_3)\subset \ker(\theta)$. Now, if $P\in
\ker(\theta)$ we deduce by a pseudo-euclidean division
that there exists $N\in \NN^\star$ such that
$$X_4^NP=Q(X_1,\ldots,X_4)(X_1X_4-X_2X_3)+R(X_2,X_3,X_4).$$
But then $R\in \ker(\theta)$ and it is obvious to check that we
have $\kk[X_2,X_3,X_4]\cap\ker(\theta)=0$.} that its
kernel is the principal ideal $(X_1X_4-X_2X_3) \subset
\kk[X_1,X_2,X_3,X_4]$. Therefore,  $\theta$ induces a graded
isomorphism of $\NN$-graded $\kk$-algebras
\begin{equation}\label{segre-iso}
\bar{\theta} : A:=\kk[X_1,X_2,X_3,X_4]/(X_1X_4-X_2X_3) \xrightarrow{\sim} S \nonumber
\end{equation}
which identifies $\PP^1\times \PP^1=\Proj(S)$ with the Segre variety
$\Sc=\Proj(A)\subset \PP^3=\Proj(\kk[X_1,X_2,X_3,X_4])$.

We are now ready to describe $\psi$. This map is of the form
\begin{eqnarray}\label{psi}
 \Sc \subset \PP^3 & \xrightarrow{\psi} & \PP^3 \\ \nonumber
(X_1:X_2:X_3:X_4) & \mapsto & (g_1:g_2:g_3:g_4)(X_1,X_2,X_3,X_4)
\end{eqnarray}
where $g_1,g_2,g_3,g_4$ are homogeneous polynomials of the same
degree in \\$\kk[X_1,X_2,X_3,X_4]$. By the graded isomorphism
$\bar{\theta}$, it follows that $\deg(\phi)=\deg(\psi)$ (we
understand co-restriction to $\Hc$) and also that the $g_i$'s must
have degree $d$. To give an algorithmic construction we just have
to determine the inverse map of $\bar{\theta}$. To do this, for
all $n\in \NN$ define the integer $k_{i,j}^{(n)}:=\max(0,n-i-j)$
and consider the map
\begin{eqnarray*}
S_n & \xrightarrow{\omega_n} & \kk[X_1,X_2,X_3,X_4]_n \\
s^iu^{n-i}t^jv^{n-j} & \mapsto &
X_1^{i+j-n+k_{i,j}^{(n)}}X_2^{n-j-k_{i,j}^{(n)}}X_3^{n-i-k_{i,j}^{(n)}}X_4^{k_{i,j}^{(n)}}
\end{eqnarray*}
(for all couples $(i,j)\in \{0,\ldots,n\}^2$). Then, we define the
map
$$\omega:=\bigoplus_{n\in \NN} \omega_n : S \rightarrow \kk[X_1,X_2,X_3,X_4]$$
which induces the inverse of $\bar{\theta}$ by passing to $A=\kk[X_1,X_2,X_3,X_4]/(X_1X_4-X_2X_3)$ (this is easy to
check). Observe also that no monomial in the image of $\omega$ is
divisible by $X_1X_4$, so our representation of the inverse of
$\bar{\theta}$ can be thought of as already reduced. Moreover,
the coefficients of the $f_i$'s and the $g_i$'s are in
correspondence: only the monomials are changed by $\omega$.

Therefore, we proved 
\begin{proposition} Defining for all $i=1,2,3,4$ the homogeneous polynomial
$$g_i(X_1,X_2,X_3,X_4):=\omega(f_i(s,u,t,v)) \in \kk[X_1,X_2,X_3,X_4]_d,$$
the map \eqref{psi} is a parametrization of the surface $\Hc\subset \PP^3$ with the property that $\deg(\psi)=\deg(\phi)$.
\end{proposition}

Furthermore, we actually proved that our initial problem, namely the
implicitization of $\phi$ in terms of syzygies, is equivalent to
the same problem with the parametrization $\psi$ which is induced
by the map
  \begin{eqnarray*}
\kk[T_1,T_2,T_3,T_4] & \xrightarrow{h} &  A\\
T_i & \mapsto & g_i(X_1,X_2,X_3,X_4).
\end{eqnarray*}
This can be summarized by the following commutative diagram, which is the algebraic translation of the diagram
 \eqref{diagram}.
$$\xymatrix{
S \ar@{<-}[r]^-{p} \ar@<3pt>[d]^-{\bar{\omega}}  & \kk[T_1,T_2,T_3,T_4] \\
A \ar@{<-}[ur]_-{h} \ar@<3pt>[u]^-{\bar{\theta}}}    
$$

This shows that the syzygies of the $f_i$'s over $S$ are in
correspondence with the syzygies of the $g_i$'s over $A$, in
particular $ker(h)=ker(p)$. Moreover, it also shows that the base
points of the parametrization $\phi$ are in one-to-one
correspondence with the base points of the parametrization $\psi$
and that their local structure (complete intersection,
multiplicity, etc.) is preserved by this correspondence. 

Another interesting remark is the following: By \cite[Theorem 2.5]{BuJo03}, we deduce that we have the equality
$$\deg(\psi)\deg(\Hc)=\deg(\Sc)d^2-\sum_{\pp \in V(g_1,\ldots,g_4)\cap\Sc\subset \PP^3} e_\pp$$
where $e_\pp$ denotes the algebraic multiplicity (in the sense of
Hilbert-Samuel). Since it is immediate to check that $\deg(\Sc)=2$
we recover the well-known formula of intersection theory (see
\cite[Prop. 4.4]{Ful} or \cite[Appendix]{Co01}):
\begin{equation}\label{deg1}
\deg(\phi)\deg(\Hc)=2d^2-\sum_{\pp \in V(f_1,\ldots,f_4)\subset \PP^1\times \PP^1} e_\pp.
\end{equation}

\medskip

Therefore, \emph{in the rest of this paper we will focus on the
implicitization of $\psi$ by means of linear syzygies}, which is a
completely equivalent problem to our initial one.

\section{The approximation complex}\label{approx}

For simplicity, we will denote by $X_i$ the classes of each variable in
the quotient ring $A=\kk[\underline{X}]/(X_1X_4-X_2X_3)$, where $\underline{X}$ stands for the sequence $X_1,X_2,X_3,X_4$. Recall that $A$ is canonically graded, each variable having weight 1. Let
$I=(g_1,g_2,g_3,g_4) \subset A$ be the ideal generated by the
$g_i$'s. We give a brief definition of the approximation complex of
cycles associated to the sequence $g_1,g_2,g_3,g_4$ over $A$. This
has been studied in depth in \cite{HSV}, see also \cite{Va94}. Under certain conditions
this complex is a free resolution of the symmetric algebra
$\Sym_A(I)$, which is one of the main motivations for its study.
Another essential feature of this complex is that - unlike the
Koszul complex - its homology depends only on the ideal
$(g_1,\ldots,g_4)$,
not on the generators $g_i$. Here is the construction: 

We consider the Koszul complex
$(K_\bullet(\underline{g},A),d_\bullet)$ associated to
$g_1,\ldots,g_4$ over $A$ and denote $Z_i=\ker(d_i)$ and
$B_i=\im(d_{i+1})$. It is of the form
$$\xymatrix@C=0.8pc{  A(-4d) \ar[r]^-{d_4} & A(-3d)^4 \ar[r]^-{d_3} & A(-2d)^6 \ar[r]^-{d_2} & A(-d)^4 \ar[r]^-{d_1} & A  }$$
where the differentials are matrices whose non-zero entries are
$\pm g_1,\ldots, \pm g_4$. We introduce new variables
$T_1,\ldots,T_4$ and set $\Zc_i= Z_i(i \cdot d) \otimes_A
A[\underline{T}]$, which we will consider as bi-graded
$A[\underline{T}]$-modules (one grading is induced by the grading
of $A$, the other one comes from setting $\deg(T_i)=1$ for all
$i$). Now the approximation complex of cycles
$(\Zc_\bullet(\underline{g},A),e_\bullet)$, or simply
$\Zc_\bullet$, is the complex $$\xymatrix@C=0.8pc{ 0 \ar[r] &
\Zc_3(-3) \ar[r]^-{e_3} & \Zc_2(-2) \ar[r]^-{e_2} & \Zc_1(-1)
\ar[r]^-{e_1} & \Zc_0 }$$ where the differentials $e_\bullet$ are
obtained by replacing $g_i$ by $T_i$ for all $i$ in the matrices
of $d_\bullet$ (note that $\Zc_4=0$, since $d_4$ is injective). It
is an important remark that  
\begin{align}\label{im} 
\im(e_1) &= \left\{ \sum_{i=1}^4 p_iT_i \ | \  p_i \in
A[\underline{T}], \sum_{i=1}^4 p_ig_i=0 \right\} \\ \nonumber
 &= \left( \sum_{i=1}^4 p_iT_i \ | \  p_i \in
A, \sum_{i=1}^4 p_ig_i=0 \right) \subset A[\underline{T}]
\end{align}
and therefore
$H_0(\Zc_\bullet) = A[\underline{T}]/\im(e_1) \simeq \Sym_A(I)$.
Note that the degree shifts indicated in the complex above are
with respect to the grading given by the $T_i$'s, while the degree shifts
with respect to the grading of $A$ are already contained in our
definition of the $\Zc_i$'s. From now on, when we take the degree
$\nu$ part  of the approximation complex, denoted
$(\Zc_\bullet)_\nu$, it should always be understood to be taken
with respect to the grading induced by $A$. Hereafter we denote by
$\mm$ the ideal $(X_1,X_2,X_3,X_4) \subset A$.

\subsection{Acyclicity criterion}

Our first concern is to show that $\Zc_\bullet(g_1,\ldots,g_4;A)$, the approximation complex of
cycles  is acyclic under suitable
assumptions. We have, similarly to \cite[{Lemma} 2]{BuCh05}, the
following
\begin{lemma}
 \label{Zacyclicity} Suppose that $I=(g_1,g_2,g_3,g_4) \subset A$
is of codimension at least $2$, and let $\Pc
  :=\Proj (A/I) \subset \Sc$. Then the following are equivalent:
\begin{enumerate}
    \item[(i)] ${\Zc}_{\bullet}$ is acyclic,
\item[(ii)] ${\Zc}_{\bullet}$ is acyclic outside $V(\mm)$,
\item[(iii)] $\Pc$ is locally defined by $3$ equations
(i.e.~locally an almost complete intersection).
\end{enumerate}
\end{lemma}
\begin{proof}
The proof is very similar to \cite[Lemma 2]{BuCh05}; the only
difference is that $A$ is not a polynomial ring here, but it is
still  a Gorenstein ring which is the main required property for
$A$. Observe that the lemma is unaffected by an extension of the
base field, so one may assume that $\kk$ is infinite.

By \cite[{Theorem} 12.9]{HSV}, we know that $\Zc_\bullet$ is
acyclic (resp. acyclic outside $V(\mm)$) if and only if $I$ is
generated by a proper sequence (resp.~$\Pc$ is locally defined by
a proper sequence). Recall that a sequence $a_1,\ldots,a_n$ of
elements in a commutative ring $B$ is a \emph{proper sequence} if
$$a_{i+1}H_{j}(a_1,\ldots,a_i;B)=0 \ \ \mathrm{for} \ i=0,\ldots,n-1 \
\mathrm{and} \ j>0,$$ where the $H_j$'s denote the homology groups
of the corresponding Koszul complex.

It is clear that (i) implies (ii). Assuming (ii), we will now deduce that
$\Pc$ is locally defined by a proper sequence. As explained in
\cite[{Lemma} 2]{BuCh05}, one can choose $h_1,h_2,h_3,h_4$ to be
sufficiently generic linear combinations of the $g_i$'s such that
\begin{itemize}
\item $(h_1,\ldots,h_4)=(g_1,\ldots,g_4) \subset A$,
 \item $h_1,h_2$ is an $A$-regular sequence, which implies that $h_1$, $h_2$, $h_3$ is a proper sequence in $A$,
\item $h_1,\ldots,h_4$ form a proper sequence outside $V(\mm)$.
\end{itemize}
By \cite[{Theorem} 1.6.16]{BH}, we have
$$H_1(h_1,h_2,h_3;A)\simeq \mathrm{Ext}^2_A(A/(h_1,h_2,h_3),A)$$
and since $A$ is Gorenstein (for it is a complete intersection),
i.e.~isomorphic to its canonical module \cite[Theorem 3.3.7]{BH},
then
\begin{equation}\label{canmod}
H_1(h_1,h_2,h_3;A)\simeq \mathrm{Ext}^2_A(A/J,A) \simeq
\omega_{A/J}
\end{equation}
outside $V(\mm)$, where $\omega_{-}$ stands for the canonical
module and $J:=(h_1,h_2,h_3)\subset A$. Since the annihilator of
$\omega_{A/J}$ over $A$  is $(J:\mm^\infty)\subset A$ (observe
that $A/J$ defines isolated points and use for instance
\cite[{Corollary} 21.3]{Eis}), we deduce that
 $h_4 \in (J:\mm^\infty)$, that is to say that $\Pc$ is locally defined by 3 equations.

Now, assume (iii). Similarly to what we did above, one can find
$h_1,\ldots,h_4$ sufficiently generic linear combinations of the
$g_i$'s so that $h_1,h_2$ is an $A$-regular sequence and
$h_1,h_2,h_3$ define $\Pc$. It follows that $h_4 \in
(J:\mm^\infty)\subset A$, where $J:=(h_1,h_2,h_3)\subset A$, and
hence \eqref{canmod} implies that $h_4$ annihilates
$H_1(h_1,h_2,h_3;A)$; it follows that $h_1,\ldots,h_4$ form a
proper sequence in $A$, so $\Zc_\bullet$ is acyclic.
\end{proof}

As soon as the base points (if there are any) of the
parametrization $\psi$ (or equivalently $\phi$) are isolated and
locally defined by 3 equations, then its associated approximation
complex of cycles is acyclic. Therefore, it can be used to compute
and represent the codimension  one part of the annihilator of the
$A[T_1,\ldots,T_4]$-module $H^0(\Zc_\bullet)$ which is nothing but
the symmetric algebra $\Sym_A(I)$. Our interest in this module is
motivated by the following

\begin{lemma}\label{annih}
Suppose that $\Pc:=\Proj(A/I)$ has dimension $\leq 0$ and is
locally defined by 3 equations. If $\eta$ is an integer such that
$$H^0_\mathfrak{m} ( \Sym_A(I) )_\nu =0 \  \text{ for all  } \nu
\geq \eta,$$ then, for all $\nu \geq \eta$ we have
$$\ann_{\kk[\underline{T}]} ( \Sym_A(I)_\nu )=\ann_{\kk[\underline{T}]} ( \Sym_A(I)_\eta ) \subseteq ker(h).$$
Moreover, the above inclusion is an equality if $\Pc$ is locally
defined by 2 equations.
\end{lemma}
\begin{proof}
For all $\nu \geq \eta$, the equality
 $$\ann_{\kk[\underline{T}]} ( \Sym_A(I)_\nu )=\ann_{\kk[\underline{T}]}( \Sym_A(I)_\eta )$$
is proven in \cite[{Proposition} 5.1]{BuJo03} for $A=\kk[\underline{X}]$. However, the same
proof can be applied without modifications to our setting: The 
key property used in the proof is the fact that the canonical map $A_1 \otimes A_n \rightarrow A_{n+1}$
is surjective and this is also valid for $A=\kk[\underline{X}]/(X_1X_4-X_2X_3)$. 
Moreover, by \eqref{im} we have
that $\ann_{\kk[\underline{T}]} ( \Sym_A(I)_\nu )\neq 0$ for $\nu
\gg 0$  if and only if $\Pc$ is locally generated by at most $3$
equations, and in this case it is clear that it is contained in
$\ker(h)$. Finally, if $\Pc$ is locally defined by at most 2
equations, meaning that $\Pc$ is locally a complete intersection,
then $I$ is of linear type outside $V(\mm)$ (use for instance
\cite[{Propositions} 4.1 and 4.5]{BuJo03}) which shows the last
claimed equality as proven in \cite[{Proposition} 5.1]{BuJo03}.
\end{proof}

In other words, if the base points of the parametrization are
isolated and locally complete intersections then certain graded
parts of the approximation complex $\Zc_\bullet$ yield a way to
compute an implicit equation of $\Hc$. Our next task is to
explicitly describe the saturation index of the symmetric algebra,
i.e.~the integer $\eta$ appearing in \textsc{Lemma} \ref{annih}. This 
will provide us with the key tool for developing the algorithm presented
in Section \ref{algo}.

\subsection{The saturation index}

For any ideal $J$ of $A$ we denote by $J^\sat$ the saturation of
$J$ with respect to the ideal $\mm$,
i.e.~$J^\sat:=(J:_A\mm^\infty) \subset A$. Also, we recall that if
$M$ is a $\NN$-graded $B$-module, where $B$ is a $\NN$-graded
ring, its initial degree is defined as $$\indeg(M):=\min\{\nu \in
\NN : M_\nu\neq 0 \} \geq 0.$$ With these notations, we have 
\begin{theorem}\label{locosymalg}
If $\Pc:=\Proj(A/I)$ is a zero-dimensional scheme (i.e.~supported
on a finite number of points, possibly zero) then
$$H^0_\mm(\Sym_A(I))_\nu=0 \qquad \forall \nu \geq 2d-1-\indeg(I^\sat).$$
\end{theorem}

The proof of this theorem is actually similar to the proof of
\cite[{Theorem} 4]{BuCh05}. The difference is that in our case the
ring $A$ is not a polynomial ring but a quotient ring. So to
validate the proof of \cite[{Theorem} 4]{BuCh05} we have to make
explicit the local cohomology and the dualizing module of $A$
which is, as a complete intersection, a Gorenstein ring (the key
property for what follows). We state these results in a little
more general case for the sake of clarity.

\begin{proposition}\label{canmod-prop}
Let $k$ be a commutative Noetherian ring and $C:=k[X_1,\ldots,X_n]$, with $n \geq 1$, which is canonically
graded by $\deg(X_i)=1$ for all $i=1,\ldots,n$. Suppose given a
homogeneous polynomial $f$ of degree $r\geq 1$ and consider the
graded quotient ring $B:=C/(f)$. The following properties hold:
\begin{itemize}
 \item $\omega_B\simeq B(-n+r)$, a graded isomorphism where $\omega_B$ stands for the canonical module of $B$,
\item $H^i_\mm(B)=0$ if $i\neq n-1$ and for all $\nu \in \ZZ$
$$H^{n-1}_\mm(B)_\nu \simeq B(-n+r)_{-\nu},$$
\item if $K_\bullet$
denotes the Koszul complex associated to a given sequence
$(a_1,\ldots,a_s)$ of homogeneous elements in $B$ of degree
$d_1,\ldots,d_s$ respectively, then we have the isomorphisms
$$H^{n-1}_\mm(K_\bullet)_\nu\simeq \Hom_{B/\mm}(K_{s-\bullet}(\sum_{i=1}^s d_i-n+r)_{-\nu},B/\mm).$$
\end{itemize}
\end{proposition}
\begin{proof} To prove the first claim, we first recall that we have $\omega_C \simeq C(-n)$.
Then, \cite[{Corollary} 3.6.14]{BH} shows that $$\omega_B \simeq
(\omega_C/f.\omega_C)(r)\simeq B(-n+r).$$

For the second claim, we recall that the local cohomology of $C$
is well-known:  $H^i_\mathfrak{m}(C)=0$ for all $i \neq n$ and
\begin{equation}\label{Cdual}
H^n_\mathfrak{m}(C)_\nu \simeq C_{-n-\nu}\end{equation}
 for all $\nu \in \ZZ$. Now, the exact sequence
$$\xymatrix@C=1pc{0 \ar[r] & C(-r) \ar[r]^-{\times f} & C \ar[r] &
B \ar[r] & 0}$$ whose long exact cohomology sequence contains the
segments
$$\xymatrix@C=1pc{ H^{j}_\mathfrak{m}(C) \ar[r] & H^{j}_\mathfrak{m}(B) \ar[r] & H^{j+1}_\mathfrak{m}(C(-r)) }$$
implies that $H^j_\mathfrak{m}(B)=0$ for all $j < n-1$ as for $j+1
< n$ both the left and the right hand side vanish. Furthermore,
the segment $$\xymatrix@C=1pc{0 \ar[r] & H^{n-1}_\mathfrak{m}(B)
\ar[r] & H^{n}_\mathfrak{m}(C(-r)) \ar[r] & H^{n}_\mathfrak{m}(C)
}$$ taken in degree $\nu$ shows $$H^{n-1}_\mathfrak{m} \left(B
\right)_\nu = \ker\left( H^n_\mathfrak{m}(C(-r))_\nu \rightarrow
H^n_\mathfrak{m}(C)_\nu \right).$$ By the self-duality of the
Koszul complex and \eqref{Cdual} this later equals exactly
$B_{-\nu-n+r}$. Finally, since $\dim(C)=n$ we have $\dim(B)=n-1$
which implies that $H^j_\mm(B)=0$ for $j>n-1$ by \cite[{Theorem}
3.5.7]{BH}.

The third claim is a direct generalization of the classical
property
$$H^{n}_\mm(K_\bullet)_\nu\simeq \Hom_{C/\mm}(K_{s-\bullet}(\sum_{i=1}^s d_i-n)_{-\nu},C/\mm).$$
The only thing which changes is the shift by $r$ in the canonical
module of $B$ and the dimension of $B$ which is $n-1$ whereas
$\dim(C)=n$.
\end{proof}

{\sc Proof of theorem \ref{locosymalg}}. We consider the
two spectral sequences associated to the double complex
$H^{\bullet}_{\mm}(\Zc_{\bullet})$:
$$\xymatrix@C=1pc@R=1pc{ 0 \ar[r] & C^0_\mathfrak{m}(\Zc_3) \ar[d] \ar[r] & C^0_\mathfrak{m}(\Zc_2)  \ar[d] \ar[r] &
C^0_\mathfrak{m}(\Zc_1) \ar[d] \ar[r] & C^0_\mathfrak{m}(\Zc_0) \ar[d] \ar[r] & 0 \\
0 \ar[r] & C^1_\mathfrak{m}(\Zc_3) \ar[d] \ar[r] &
C^1_\mathfrak{m}(\Zc_2)  \ar[d] \ar[r] & C^1_\mathfrak{m}(\Zc_1)
\ar[d] \ar[r] & C^1_\mathfrak{m}(\Zc_0) \ar[d] \ar[r] & 0 \\
& \vdots \ar[d] & \vdots \ar[d] & \vdots \ar[d] & \vdots \ar[d] \\
0 \ar[r] & C^4_\mathfrak{m}(\Zc_3)  \ar[r] &
C^4_\mathfrak{m}(\Zc_2)  \ar[r] & C^4_\mathfrak{m}(\Zc_1)
 \ar[r] & C^4_\mathfrak{m}(\Zc_0) \ar[r] & 0}$$ They both converge to the
hypercohomology of $\Zc_{\bullet}$. One of them stabilizes at
level two with:
$$
{_{2}{'}E}^{p}_{q}={_{\infty}{'E}}^{p}_{q}=\left\{
\begin{array}{cl}
H^{p}_{\mm}(H_{q}(\Zc_{\bullet})) &\hbox{for}\ p=0,1\
  \hbox{and}\ q>0 \\
  H^{p}_{\mm}(\Sym_{A}(I)) &\hbox{for}\ q=0\\
0 &\hbox{else}.\\
\end{array}\right.
$$
and the other one gives at level one:
$$
{_{1}{''E}}^{p}_{q}=H^{p}_{\mm}(Z_{q})[qd]\otimes_{A}A[\underline{T}](-q).
$$
As explained in \cite[{Theorem} 4]{BuCh05}, the comparison of
these two spectral sequences and  \cite[{Lemma} 1]{BuCh05}
show\footnote{Note that \cite[{Lemma} 1]{BuCh05} can be applied
verbatim in our case (modulo some little change on the degree
shifts that we will describe below) because of
\textsc{Proposition} \ref{canmod-prop}.} that the module
$H^0_\mm(\Sym_A(I))_\nu$ vanishes as soon as $(_1{''}E^{p}_p)_\nu$
vanishes for $p=2,3$. Moreover, setting
$\hbox{---}^{\star}:=\Homgr_{A}(\hbox{---},A/\mm )$, we have the
graded isomorphisms
$$_{1}{''}E_3^3 \simeq  (A/I)^\star[2-d]\otimes_A A[\underline{T}](-3)$$
and
$$_{1}{''}E_2^2 \simeq (I^\sat/I)^\star[2-2d]\otimes A[\underline{T}](-2).$$
It follows that $(_1{''}E^{2}_2)_\nu$ and $(_1{''}E^{3}_3)_\nu$
vanish simultaneously if
$$\nu > \min(d-2,2d-2-\indeg(I^\sat/I)).$$ This is true whenever $\nu \geq \nu_0:=2d-1-\indeg(I^\sat)$, since
$\min(d,\indeg(I^\sat/I))=\indeg(I^\sat)$. \hfill
$\square$

\begin{remark} Since $I$ is generated in degree $d$ and $I\subset I^\sat$ we have the inequality
$0\leq \indeg(I^\sat) \leq d$. It follows that
$$ d-1 \leq 2d-1-\indeg(I^\sat) \leq 2d-1.$$
The lower bound is reached whenever the ideal $I$ is saturated
(meaning $I=I^\sat$) and the higher bound corresponds to the
abscence of base points of the parametrization.
\end{remark}

\subsection{The main result}\label{mainresult}

We now have all the tools necessary at our disposal and can
proceed to the main result of this paper. But before, recall that
there are two distinct notions of multiplicity for a base point
$\pp \in V(I)\cap \Sc \subset \PP^3$: the algebraic multiplicity
denoted $e_\pp$ and the geometric multiplicity denoted $d_\pp$
(see for instance \cite[\S2.2]{BuJo03} for more details).

\begin{theorem}\label{mainth}
Assume that $\dim \Pc:=\Proj(A/I) \leq 0$ and that $\Pc$ is
locally an almost complete intersection (i.e.~locally defined by 3
equations). Then, for every integer
$$\nu \geq \nu_0:=2d-1-\indeg(I^{\sat})$$
the determinant $D$ of the complex $(\Zc_\bullet)_\nu$ of
$\kk[\underline{T}]$-modules (which is unique up to multiplication by a non-zero constant in $\kk$) is a non-zero
homogeneous element in $\kk[\underline{T}]$, independent of $\nu \geq \nu_0$ and of degree 
$$2d^2-\sum_{\pp\in V(I)\cap \Sc \subset \PP^3} d_\pp$$
such that $D=F^{\deg(\psi)}G$ where $F$ is the implicit equation
of $\Hc$, $G$ is coprime with $F$ and $\deg(G)=\sum_{\pp\in V(I)\cap \Sc}
(e_\pp-d_\pp)$.

Moreover, $G\in \kk\setminus \{ 0 \}$ if and only if $\Pc$ is locally a
complete intersection (i.e.~locally defined by 2 equations).
\end{theorem}

\begin{proof}
First of all, observe that $D$ is independent of $\nu$ by {\sc
theorem} \ref{locosymalg}. It is an homogeneous element of
$\kk[\underline{T}]$ because $(\Zc_\bullet)_\nu$ is a graded
complex of $\kk[\underline{T}]$-modules and it is non-zero because
$\Pc$ is locally an almost complete intersection, a fact we already used in 
\textsc{Lemma} \ref{annih}.
%

The computation of $\deg(D)$ can be done as in {\cite[Theorem
4]{BuCh05}}: For $\nu \gg 0$ we have
$$\deg(D)=\dim(Z_1)_{\nu+d}-2\dim(Z_2)_{\nu+2d}+3\dim(Z_3)_{\nu+3d}.$$
In the case where all the $H_i$'s, with $i>0$, vanish then
$\deg(D)=2d^2$. If $H_1$ and $H_2$ are non-zero, then they
contribute to the above quantity for
\begin{multline}
\dim (H_1)_{\nu+d}-\dim (H_2)_{\nu+d} - 2 \dim (H_2)_{\nu+2d} \\
= \dim (H_0)_{\nu+d} -2\dim (H_2)_{\nu+2d} = -deg \Pc
 \end{multline}
where we assume that $\nu \gg 0$, since $H_2\simeq \omega_{A/I}$.
Therefore, we deduce that
\begin{equation}\label{deg2}
\deg(D)=2d^2-\deg\Pc=2d^2-\sum_{\pp\in V(I)\cap \Sc \subset \PP^3} d_\pp.
\end{equation}

Now, setting $\qc:=\ker(h)$ and using standard properties of
determinants of complexes we compute
\begin{align*}
[\det((\Zc_\bullet)_\nu)] &= \mathrm{div}(H_0(\Zc_\bullet)) \\
&= \mathrm{div}(\Sym_A(I)_\nu) \\
&= \sum_{ \pp \text{ prime}, \ \mathrm{codim}(\pp)=1 } \length((\Sym_A(I)_\nu)_\pp)\cdot [\pp]\\
&= \length((\Sym_A(I)_\nu)_\qc)\cdot [\qc] + \cdots. \\
\end{align*}
Since $\length((\Sym_A(I)_\nu)_\qc)=\deg(\psi)$ as proved in
{\cite[Theorem 5.2]{BuJo03}}, we deduce that $D=F^{\deg(\psi)}G$
where $G$ does not divide $F$.

Finally, using equations \eqref{deg1} and \eqref{deg2} we deduce that $$\deg(G)=\sum_{\pp \in V(I)\cap \Sc} (e_\pp
-d_\pp),$$
 and it is well-known that $e_\pp\geq d_\pp$ with equality
if and only if the point $\pp$ is locally a complete intersection.
\end{proof}

Recall that the determinant of the complex $(\Zc_\bullet)_\nu$ can
either be obtained as an alternating product over some sub-determinants of
the matrices appearing in the complex or as a $gcd$ of maximal
minors of the first map in the $(\Zc_\bullet)_\nu$-complex (we
will explicitly construct this matrix $M$ in the next section).
One can either compute this $gcd$ directly or factorize one of the
maximal minors, however, both methods are computationally
expensive (as all existing implicitization methods). 

From a practical point of view, it might be interesting to avoid the
actual computation of an implicit equation and use instead the matrix $M$
as an implicit representation of the surface, since it is more
compact and much easier to compute. To give an example, let us
suppose that we are in the case of locally complete intersection base points. Then if we
want to decide if a given point $P$ lies on the surface there is no
need to compute the implicit equation: It suffices to evaluate $M$
in this point, as the rank of $M$ drops if and only if $P$ belongs to
the surface. 

This is due to the fact that for a commutative ring $R$ and a
morphism $\alpha : R^m \rightarrow R^n$ with $m \geq n$ we always have
$$\ann_R(\coker(\alpha))^n \subseteq I_n(\alpha) \subseteq
\ann_R(\coker(\alpha))$$ where $I_n(\alpha)$ denotes the ideal
generated by the maximal minors of the matrix of $\alpha$,
i.e.~the principal Fitting ideal of $\alpha$ (see for instance
\cite[{Proposition} 20.7]{Eis}). Ours is the special case
$R=\kk[\underline{T}]$ and $\alpha$ is the first map in
$(\Zc_\bullet)_\nu$, i.e. the one induced by $e_1$, and hence
$coker(\alpha)=\Sym_A(I)_\nu$. Geometrically, this means that the
maximal minors of $M$ define the hypersurface $\Hc$ by
\textsc{Lemma} \ref{annih}, and consequently, the points for which
the rank of $M$ drops are
exactly those belonging to $\Hc$. 

Similarly, other problems arising from applications might be
solved by direct computations using the matrix representation
without the (expensive) transition to the implicit equation.

\section{Algorithm}\label{algo}

In order to show explicitly how the theoretical results from the
previous sections are used in practice, we formulate an algorithm
for the actual computation of the matrix representing the implicit
equation. It is efficient and easy to implement, as it consists
basically of the resolution of a linear system. We give only the
essential steps, see \cite[Section 3]{BuCh05} for a more detailed
description of a very similar algorithm.
\begin{itemize}
\item Given four bi-homogeneous polynomials $f_1,f_2,f_3,f_4$ of
degree $d$, define the homogeneous polynomials $g_1,g_2,$ $g_3,g_4
\in A=\kk[\underline{X}]/(X_1X_4-X_2X_3)$ of the same degree by
setting $g_i = \omega(f_i)$, where $\omega$ is the isomorphism
defined in Section \ref{segre}.

\item Find the solution space $W$ of the linear system (over
$\kk$) defined by $$\sum_{i \in \{1,\ldots,4\}} a_i g_i = 0$$
where $(a_1,a_2,a_3,a_4) \in (A_{\nu_0})^4$  and $\nu_0
=2d-1-\indeg(I^\sat)$, i.e. one writes the equation with respect
to a basis of $A_{\nu_0+d}$ and compares the coefficients. $W$ is
represented by a $\dim_\kk (A_{\nu_0+d}) \times
4\dim_\kk(A_{\nu_0})$-matrix $N$, where the first $k :=
\dim_\kk(A_{\nu_0})$ columns represent the coefficients of $a_1$,
the next $k$ coefficients $a_2$, etc.

\item For $i \in \{1,\ldots,4\}$, let $M_i$ be the $k \times
k$-matrix $T_i \cdot \mathrm{Id}_k$. Then
$$M := N \cdot \begin{pmatrix} M_1 \\ \vdots \\ M_4 \end{pmatrix}$$ is a
matrix of the first map of the graded part $({\Zc_\bullet})_{\nu_0}$of the approximation complex.
\end{itemize}

As we proved, in the case where the base points of the
parametrization $\phi$ are isolated and locally complete
intersections, $M$ represents the surface $\Hc$. Also, the $gcd$
of the maximal minors (of size $k$) of $M$ equals its implicit
equation.

\subsubsection*{An illustrative example}

We now present an example to illustrate our method, which provides
a matrix-based representation of the implicit equation of $\Hc$ by
means of the linear syzygies of its parametrization $\phi$ (or more
precisely, of $\psi$). It should be emphasized that all the following
computations are presented in order to explore in detail our
approach and are not all required to get the expected matrix-based
representation. Our code is written for {\tt Macaulay2} (see
\cite{M2}), in which one can easily compute all the terms and maps of
the approximation complex.

Consider the following example taken from \cite[Example
4.16]{AHW05}:
\begin{verbatim}
S=QQ[s,u,t,v];
d=2;
f1=u^2*t*v+s^2*t*v
f2=u^2*t^2+s*u*v^2
f3=s^2*v^2+s^2*t^2
f4=s^2*t*v 
F=matrix{{f1,f2,f3,f4}}
\end{verbatim}
Note that the interested reader can experiment with his own
example just by changing the above definitions of the polynomials
$f_1,f_2,f_3,f_4$ giving the parametrization.

The first thing to do is to use the isomorphism $\bar{\theta}$
to switch from $S$ (note that the ring $S$ defined in the above
command is not exactly the ring $S$ we have introduced in Section
\ref{segre}) to the ring $A$:
\begin{verbatim}
SX=S[x1,x2,x3,x4]
F=sub(F,SX)
ST={}; X={};
for i from 0 to d do (
  for j from 0 to d do (
    k=max(0,d-i-j);
    ST=append(ST,s^i*u^(d-i)*t^j*v^(d-j));
    X=append(X,
      x1^(i+j-d+k)*x2^(d-j-k)*x3^(d-i-k)*x4^(k));
    )
  )
ST=matrix {ST}; ST=sub(ST,SX); X=matrix {X};
(M,C)=coefficients(F,Variables=>
              {s_SX,u_SX,t_SX,v_SX},Monomials=>ST)
G=X*C -- this is the parametrization, but in SX
A=QQ[x1,x2,x3,x4]/(x1*x4-x2*x3)
r=map(A,SX,{x1,x2,x3,x4,0,0,0,0})
G=r(G);
G=matrix{{G_(0,0),G_(0,1),G_(0,2),G_(0,3)}}
\end{verbatim}
The matrix $G$ is the matrix (with entries in $A$) of the
parametrization $\psi$ from the Segre variety $\Sc$ to $\PP^3$.
One should note that the quotient ring $A$ is a very simple
quotient ring: essentially, computations in $A$ can be done in
$\kk[X_1,\ldots,X_4]$ modulo the substitution of $X_1X_4$ by
$X_2X_3$. Moreover, bases for $A$ in any given degrees can easily
be pre-computed since they do not depend on the given
parametrization $\phi$.

We can now define the terms of the approximation complex of cycles
$\Zc_\bullet$:

\begin{verbatim}
Z0=A^1;
Z1=kernel koszul(1,G);
Z2=kernel koszul(2,G);
Z3=kernel koszul(3,G);
\end{verbatim}
As we already remarked, $\Zc_4 = 0$. Define the integer

\begin{verbatim}
nu=2*d-1
\end{verbatim}
We can compute the Euler characteristic of $(\Zc_\bullet)_\nu$ and
check that it is zero with the command
\begin{verbatim}
hilbertFunction(nu,Z0)-hilbertFunction(nu+d,Z1)+
hilbertFunction(nu+2*d,Z2)-hilbertFunction(nu+3*d,Z3)
\end{verbatim}
and also compute the degree of $D$, the determinant of the complex
$(\Zc_\bullet)_\nu$, with the command
\begin{verbatim}
hilbertFunction(nu+d,Z1)-2*hilbertFunction(nu+2*d,Z2)
+3*hilbertFunction(nu+3*d,Z3)
\end{verbatim}
This number equals the degree of $\Hc$ if all the base points, if any, form
locally a complete intersection. In this example, we find degree 7.

At this step, one can try to lower the integer $\nu$ according to
{\sc Theorem} \ref{mainth}; to this end we compute the degrees of
the generators of the saturation of the ideal $(g_1,\ldots,g_4)$:
\begin{verbatim}
degrees gens saturate(ideal G,ideal(x1,x2,x3,x4))
\end{verbatim}
Since, in this example, the smallest degree is 1 we can redefine
\begin{verbatim}
nu=2*d-2
\end{verbatim}
and we can re-check the Euler characteristic and the degree of the
determinant of $(\Zc_\bullet)_\nu$.

We can now compute the matrix of the first map of
$(\Zc_\bullet)_\nu$, that is to say the matrix of linear syzygies
of $g_1,\ldots,g_4$ which represents $\Hc$:
\begin{verbatim}
R=A[T1,T2,T3,T4]
G=sub(G,R);
Z1nu=super basis(nu+d,Z1);
Tnu=matrix{{T1,T2,T3,T4}}*substitute(Z1nu,R);
(m,M)=
 coefficients(Tnu,Variables=>{x1_R,x2_R,x3_R,x4_R},
             Monomials=>substitute(basis(nu,A),R));
\end{verbatim}
The matrix {\tt M} is the desired matrix, and it is of size $9\times
12$.

\section{Comments and conclusion}

We have presented a new approach to compute an implicit
representation in terms of linear syzygies for a surface in $\mathbb{P}^3$ parametrized by 
bi-ho\-mo\-ge\-neous polynomials of bi-degree $(d,d)$, $d\geq 1$, under the
assumption that the base points are isolated and locally complete intersections.
 This result, along with the similar
ones for parametrizations over the projective plane, shows that in
many cases it is not necessary to use quadratic syzygies in
order to represent the implicit equation of a surface. 

We should point out that this method has the advantages of being
valid in a very general setting (we have neither assumed
birationality nor made other additional assumptions on the
parametrization) and of working well in the presence of base
points. Furthermore, the matrix representing the surface can be
computed in a very efficient way. \\

It would be nice if we could use the same method for mixed degrees
as well, i.e. consider parametrizations by bi-ho\-mo\-ge\-neous
polynomials of bi-degree $(d_1,d_2)$ with $d_1,d_2\geq 1$. Let us
discuss some ideas on how to generalize to the mixed case:
\begin{itemize}
 \item Putting weights on the variables in $S$ will not give us good
properties for $S$, for instance $S$ will not be generated by
$S_1$ as an $S_0$ algebra in general. \item Considering the
bi-degree $(\max(d_1,d_2),\max(d_1,d_2))$ is not possible because
 it introduces a base point locus of positive dimension and we will lose the acyclicity of the approximation
 complex.
\item One way to come back to unmixed bi-degree is to make the
substitutions
$$s \leftarrow s^{\lcm(d_1,d_2)/d_1} \text{ and } t
\leftarrow t^{\lcm(d_1,d_2)/d_2}.$$ Everything works fine in this
case, but we are not representing $F^{\deg(\psi)}$, but
$F^{\deg(\psi)\lcm(d_1,d_2)/\gcd(d_1,d_2)}$ which is is not
optimal, as it increases the size of the matrices involved. For
instance, we could treat Example 10 from \cite{KD06} in this way.
It is a surface of bi-degree (2,3) defined by \begin{eqnarray*} f_1 &=&(t+t^2)(s-1)^2+(1+st-s^2t)(t-1)^2 \\
f_2 &=&  (-t-t^2)(s-1)^2 + (-1+st+s^2t)(t-1)^2\\
f_3 &=&  (t-t^2)(s-1)^2 + (-1-st+s^2t)(t-1)^2\\
f_4 &=& (t+t^2)(s-1)^2 + (-1-st-s^2t)(t-1)^2
\end{eqnarray*} By replacing $s$ by $s^3$ and $t$ by $t^2$, we obtain
 a parametrization of bi-degree (6,6) and $F^6$ can indeed be
computed in degree $\nu \geq 2 \cdot 6-1 - 6 = 5$ of the
approximation complex as the $gcd$ of the maximal minors of a $42
\times 36$-matrix, whereas in the original paper it was computed
as the determinant of a $5 \times 5$-matrix.
\end{itemize}

Therefore, it seems that the tools we used above (and which work
well for unmixed bi-degree) are not well-suited for this more
general case and that it might be necessary to take into account
the bi-graded structure of S in order to devise a method that is
adapted to mixed bi-degrees. We hope to develop this in the near
future.

\section{Acknowledgments}
The authors have been partially supported by the French ANR
``Gecko".

\bibliographystyle{abbrv}

\end{document}